\documentclass[12pt]{article}

\usepackage{multicol}
\usepackage{amssymb,amsthm,amsmath}
\usepackage[usenames,dvipsnames]{xcolor}
\usepackage{tikz}
\usepackage{mathabx}
\allowdisplaybreaks
\usepackage{color,graphicx,epsfig}
\usepackage{mathrsfs}
\usepackage{enumerate} 
\usepackage{enumitem} 
\usepackage{esvect}

\usepackage{tikz}

\newcommand{\NASM}{N\hspace{-.5mm}ASM}

\newcommand{\comment}[1]{}

\definecolor{teal}{RGB}{0,128,128}
\definecolor{darkpurple}{RGB}{128,0,128}

\usepackage[titletoc,toc,title]{appendix}

\newtheorem{theorem}{Theorem}[section]

\newtheorem{lemma}[theorem]{Lemma}
\newtheorem{cor}[theorem]{Corollary}
\newtheorem{guess}[theorem]{Conjecture}
\theoremstyle{definition}

\theoremstyle{definition}

\newtheorem{ex}[theorem]{Example}

\def \Z {\mathbb Z}

\def \h {\mathbf h}
\def \k {\mathbf k}

\def \hh {\mathbf h}

\def \kk {\mathbf k}

\def \uu {\mathbf u}
\def \vv {\mathbf v}

\title{A note on near alternating sign matrices with prescribed row and column weights}

\author{
T.\ Traetta\footnotemark[1]
}
\date{\vspace{-5ex}}

\begin{document}
\maketitle

\footnotetext[2]{DICATAM, Universit\`{a} degli Studi di Brescia, Via Branze 43, 25123 Brescia, Italy. E-mail: tommaso.traetta@unibs.it}

\begin{abstract} 
We study near alternating sign matrices with prescribed row and column
weights. After recalling the basic definitions and the necessary conditions
coming from the Gale--Ryser theorem, we give a graph-theoretic
characterization of the existence problem. More precisely, we show that a
near alternating sign matrix with prescribed row and column weight sequences exists
if and only if there exists a binary array with the same weights whose
associated adjacency graph is bipartite. This reformulation shows that the
Gale--Ryser conditions are not sufficient in general. We then formulate a
natural conjecture up to permutations of the row and column weight
sequences, and present some positive results, based on convex binary
arrays and composition constructions.
\end{abstract}

\section{Introduction}
Given an $m\times n$ matrix $A$ with entries from $\Z$, we refer to the number of nonzero entries of a row (or a column) of $A$ as its \emph{weight}. We also denote by $w_r(A)$ and $w_c(A)$ the sequences of row and column weights of $A$, respectively. 
Clearly, if $w_r(A) = (h_1, \ldots, h_m)$ and  $w_c(A) = (k_1, \ldots,k_n)$, we must have that $\textstyle{\sum_{i=1}^{m} h_i=\sum_{j=1}^{n} k_j}$.
We denote by $\Phi(m,n)$ and $\Psi(m,n)$ the sets of all $m\times n$ arrays with entries from $\{0,1\}$ and $\{0, \pm 1\}$, respectively; an array in 
$\Phi(m,n)$ is usually called a \emph{binary array}. Also,
$\Phi(m,n; \mathbf{h}, \mathbf{k})$ and $\Psi(m,n; \mathbf{h}, \mathbf{k})$ denote the subsets of $\Phi(m,n)$ and $\Psi(m,n)$, respectively, consisting of the arrays with row-weight sequence $\mathbf{h}$ and column-weight sequence $\mathbf{k}$. 

An array $A\in \Psi(m,n)$ is said to be a \emph{near alternating sign matrix} (NASM) if the nonzero entries of each row and each column alternate in sign. 
Letting $\mathbf{h}=w_r(A)$ and $\mathbf{k}=w_c(A)$, we say that
$A$ is a NASM$(m,n; \mathbf{h}, \mathbf{k})$
or simply a NASM$(m,n)$. 
Finally, if $\mathbf{h}=(h, \ldots, h)$ and $\mathbf{k}=(k, \ldots, k)$, we say that $A$ is uniform and write NASM$(m,n; h,k)$. 

Since we are concerned with the existence of a NASM having prescribed row and column weights, irrespective of their ordering, 
we write NASM$(m,n;$ $[\mathbf{h}], [\mathbf{k}])$ 
to denote any NASM$(m,n; \mathbf{h}', \mathbf{k}')$ 
such that $\mathbf{h}'$ is a permutation of $\mathbf{h}$ and $\mathbf{k}'$ is a permutation of $\mathbf{k}$.

\begin{ex}
Here is a \NASM$(3,3;(3,3,2),(3,3,2))$.
\[
A=
\begin{pmatrix}
 1&-1& 1\\
-1& 1&-1\\
 1&-1& 0
\end{pmatrix}
\]
\end{ex}

Near alternating sign matrices were first considered in
\cite{BruKim} (as a generalization of the classic alternating sign matrices), although the terminology was introduced in \cite{MT} were the authors (motivated by the application in constructing generalized Heffter arrays) were interested in building them with given row-weight and column-weight sequences. We recall the following well-known result by Gale and Ryser \cite[Theorem 7.7.4]{DJ} on the existence of a binary array with given row and column weights.

\begin{theorem}\label{GR}
	  Let $\hh=(h_1, \ldots, h_m)$ and $\kk=(k_1, \ldots,k_n)$ be two sequences of positive integers with $0\leq h_i\leq n$, $0\leq k_j\leq m$, and such that $\textstyle{\sum_{i=1}^{m} h_i=\sum_{j=1}^{n} k_j}$. 
	  There exists an $m\times n$ binary array 
	  $A$ with $w_r(A)=\hh$ and $w_c(A)=\kk$ if~and~only~if 
  \begin{equation}\label{GHA:nec2}
    \text{$\textstyle \sum_{i=1}^m \min (h_i,u) \geq \sum_{j=1}^u k'_j$, for every $1\leq u\leq n$}.
  \end{equation}
  where $(k'_1, \ldots, k'_n)$ is the decreasing reordering of $\kk$.
\end{theorem}

  Taking the entries of a NASM$(m,n;\hh,\kk)$ in absolute value yields a binary array with the same row and column weights. Therefore, by Theorem \ref{GR}, it follows that $\hh$ and $\kk$ must satisfy condition \eqref{GHA:nec2}. In the following corollary, we collect the trivial necessary conditions for the existence of a NASM$(m,n;\hh,\kk)$.
  
  \begin{lemma} \label{nec}
  If there exists a NASM$(m,n;\hh,\kk)$, where $\hh=(h_1, \ldots, h_m)$ and $\kk=(k_1, \ldots,k_n)$, then 
  \begin{enumerate}
    \item $0\leq h_i \leq n$ and $0\leq k_j \leq m$, for every $1\leq i\leq m$ and 
    $1\leq j\leq n$,
    \item $\textstyle{\sum_{i=1}^{m} h_i=\sum_{j=1}^{n} k_j}$, and
    \item $\textstyle \sum_{i=1}^m \min (h_i,u) \geq \sum_{j=1}^u k'_j$, for every $1\leq u\leq n$, where $(k'_1, \ldots, k'_n)$ is the decreasing reordering of $\kk$.
  \end{enumerate}
  \end{lemma}
    
  In this paper, we show that the necessary conditions stated in Lemma~\ref{nec} are not sufficient for the existence of a $\mathrm{NASM}(m,n;\mathbf{h},\mathbf{k})$ with prescribed row and column weight sequences $(\mathbf{h},\mathbf{k})$. To this end, we translate the existence problem for NASMs into the problem of determining the existence of a binary array with chromatic number~2 (see Section~2).

We then raise the question of whether the conditions of Lemma~\ref{nec} might nevertheless be sufficient for the existence of a $\mathrm{NASM}(m,n;[\mathbf{h}],[\mathbf{k}])$. In Section~3, we recall the notion of frame for a NASM and use it to develop some composition constructions. We conclude the paper with some remarks and open problems.

\section{NASMs as binary arrays with chromatic number $2$}
Given a binary array $A=(a_{i,j})\in \Phi(m,n)$, we define a graph $\Gamma(A)$ as follows.
The vertices of $\Gamma(A)$ are the entries of $A$ equal to $1$, that is, 
\[
V(\Gamma(A)) = \{(i,j)\in \{1,\ldots,m\}\times\{1,\ldots,n\}\mid a_{ij}=1\}.\]
Two vertices
are adjacent whenever the corresponding entries of $A$ are consecutive
nonzero entries in the same row or in the same column. Equivalently, after deleting the zero entries of $A$, we join two ones if
they are consecutive horizontally or vertically.

\begin{ex} A binary array $A$ and its associated graph $\Gamma(A)$.\\
\begin{minipage}[c]{0.48\textwidth}\centering
  $A=\left(
\begin{matrix}{}
1 & 1 & 1 & 0\\[2pt]
0 & 1 & 0 & 1\\[2pt]
1 & 0 & 1 & 1\\[2pt]
1 & 1 & 1 & 0\\ 
\end{matrix}\right)$
\end{minipage}
\begin{minipage}[c]{0.48\textwidth}\centering
\includegraphics[scale=0.4]{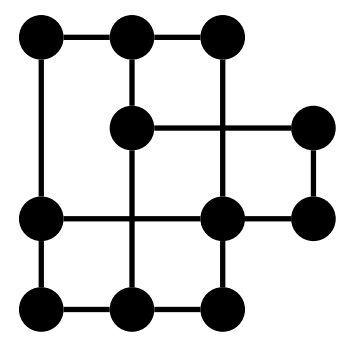}
\end{minipage}
\end{ex}

We say that 
a binary array $A$ is  $c$-colorable if $\Gamma(A)$ admits a
proper vertex-coloring with $c$ colors. 
Recall that the minimum $c$ for which $\Gamma(A)$ is $c$-colorable is the chromatic number
$\chi(\Gamma(A))$ of $\Gamma(A)$. Also, $\Gamma(A)$ is bipartite if and only if 
$\chi(\Gamma(A))\leq 2$.

We refer to $\chi(\Gamma(A))$ as the chromatic number $\chi(A)$ of $A$, and show that the chromatic number of a binary array is at most $3$.

\begin{lemma}
If $A$ is a binary array, then $\chi(A)\leq 3$.
\end{lemma}

\begin{proof}
Let $A\in \Phi(m,n)$. We describe a coloring procedure that produces a proper vertex-coloring of $\Gamma(A)$ using at most three colors.
For each $j\in\{1,\ldots,n\}$, let $\Gamma_j$ denote the subgraph of $\Gamma(A)$ induced by the vertices corresponding to the nonzero entries of $A$ in the first $j$ columns. We proceed by induction on $j$.

Since $\Gamma_1$ is a path, it admits a proper $2$-coloring. Assume now that $\Gamma_u$ has been properly colored with at most three colors, for some $1\leq u<n$. Consider the graph
$\Gamma'=\Gamma_{u+1}\setminus E(\Gamma_u)$,
obtained from $\Gamma_{u+1}$ by removing all edges of $\Gamma_u$.
All vertices of $\Gamma'$ are already colored, except those in
$W=V(\Gamma_{u+1})\setminus V(\Gamma_u)$,
which correspond to the nonzero entries in column $u+1$ of $A$. 
The subgraph induced by $W$ is a path. Moreover, each vertex of $W$ is adjacent to at most one previously colored vertex, namely the closest nonzero entry to its left in the same row, if such an entry exists. Hence, before coloring the vertices of $W$, each of them has at least two admissible colors among the three available colors.

We now color the vertices of $W$ sequentially along the path induced by $W$. At each step, the current vertex has at most one previously colored neighbor in $W$. Since it has at least two admissible colors, we can choose one that is different from the color assigned to that previously colored neighbor. Thus the coloring remains proper.
This extends the coloring of $\Gamma_u$ to a proper $3$-coloring of $\Gamma_{u+1}$. By induction, $\Gamma(A)$ admits a proper coloring with at most three colors, and therefore $\chi(A)\leq 3$.

\end{proof}

The next theorem is the fundamental link between near alternating sign
matrices and binary arrays.

\begin{theorem}\label{thm:bipartite}
There exists a \NASM$(m,n;\h,\k)$ if and only if there exists a binary
array $A\in\Phi(m,n;\h,\k)$ such that $\Gamma(A)$ is bipartite.
\end{theorem}
\begin{proof}
Suppose first that $M$ is a \NASM$(m,n;\h,\k)$, and let $A$ be the binary array obtained by taking the entried of $M$ in absolute value. Then
$A\in\Phi(m,n;\h,\k)$. Color each vertex of $\Gamma(A)$ with a $+$ or $- $ according to the
sign of the corresponding entry of $M$. Since the nonzero entries of every
row and every column of $M$ alternate in sign, two adjacent vertices of
$\Gamma(A)$ receive different colors. Hence $\Gamma(A)$ has a proper vertex coloring with at most two colors, that is, $\Gamma(A)$ is bipartite.

Conversely, assume that $A\in\Phi(m,n;\h,\k)$ and that $\Gamma(A)$ is
bipartite. Choose a bipartition of the vertex set of $\Gamma(A)$. Replace
the entries of $A$ corresponding to one part by $+1$, and those
corresponding to the other part by $-1$. Leave all zero entries unchanged.
The resulting array belongs to $\Psi(m,n)$, has row and column weights
$\h$ and $\k$, and its nonzero entries alternate in every row and every
column. Hence it is a \NASM$(m,n;\h,\k)$.
\end{proof}

Theorem~\ref{thm:bipartite} shows that, besides the existence of a binary 
array $A$ with the prescribed weights, one needs the graph associated to $A$ is bipartite in order to build a NASM with the same size and weights of $A$.  The following example shows that the
Gale--Ryser conditions alone are not sufficient for the existence of a NASM with prescribed row and column weight sequences.

\begin{ex}
Let $\h=\k=(3,2,3)$.
There is a unique binary array in $\Phi(3,3; \h,\k)$, namely
$A=
\begin{pmatrix}
1&1&1\\
1&0&1\\
1&1&1
\end{pmatrix}$, whose associated graph is 

\begin{minipage}[c]{.5\textwidth}
  \ \hfill $\Gamma(A) =$
\end{minipage}
\begin{minipage}[c]{.5\textwidth}
  \includegraphics[scale=0.25]{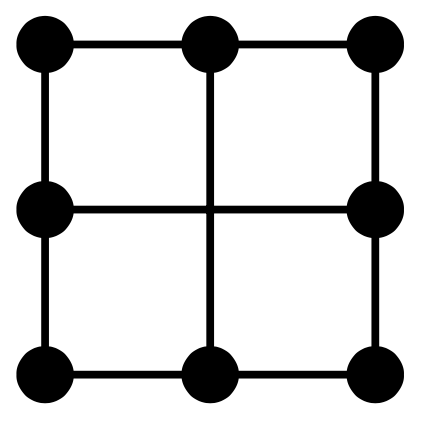}
\end{minipage}
\vspace{4mm}

\noindent
Since $\Gamma(A)$ clearly contains a $5$-cycle,  it is not bipartite. Therefore,
by Theorem~\ref{thm:bipartite}, there is no
\NASM$(3,3;(3,2,3),(3,2,3))$.

On the other end, the matrix
\[
A=
\begin{pmatrix}
 1&-1& 1\\
-1& 1&-1\\
 1&-1& 0
\end{pmatrix}
\]
is a NASM$(3,3;[\h],[\k])$ whose row-weight and column-weight sequences are both equal to $(3,3,2)$, a permutation of  $\h=\k$.
\end{ex}

The above example can be readily extended to yield an infinite family of prescribed row and column weight sequences, say $(\h,\k)$, satisfying Lemma \ref{nec}, for which no $\mathrm{NASM}(m,n;\mathbf{h},\mathbf{k})$ exists. However, if one disregards the order in which the weights occur among the rows and columns, a NASM with the prescribed multisets of row and column weights may still exist, even when non-existence has been established for particular orderings of those multisets. In other words, the necessary conditions stated in Lemma~\ref{nec} may still be sufficient for the existence of a
\[
\mathrm{NASM}(m,n;[\mathbf{h}],[\mathbf{k}]).
\]
This motivates the following conjecture.

\begin{guess}\label{conj:main}
Let $\h$ and $\k$ be two sequences of positive integers  satisfying the conditions of Lemma \ref{nec}. Then there exists a \NASM$(m,n;[\h],[\k])$.
\end{guess}

This conjecture holds in some special cases: for instance, this is known when
$\h$ and $\k$ are constant \cite[Theorem 3.6]{MT} or when $\h$ and $\k$ satisfy a very specific set of conditions, see \cite[Theorem 3.8]{MT}.

In the remaining part of this section, we show that ``convex'' arrays  immediately yield near alternating sign matrices. 
We recall that a binary array is called \emph{convex} if the ones are consecutive in each
row and in each column. 

\begin{ex} Here is a convex array in $\Phi(m,n; (2,4,1), (1,3,2,1))$.
\[A=
\begin{pmatrix}
0&1&1&0\\
1&1&1&1\\
0&1&0&0
\end{pmatrix}
\]
\end{ex}

\begin{theorem}\label{thm:convex}
If there exists a convex array
$A\in\Phi(m,n;\h,\k)$, then there exists a \NASM$(m,n;\h,\k)$.
\end{theorem}
\begin{proof}
Let $J$ be the $m\times n$ all-one matrix. The graph $\Gamma(J)$ is the
Cartesian product of two paths, and is therefore
bipartite.
Since $A$ is convex, the graph $\Gamma(A)$ is an induced subgraph of
$\Gamma(J)$. Hence $\Gamma(A)$ is bipartite. The result then follows from
Theorem~\ref{thm:bipartite}.
\end{proof}

\begin{ex} Here is a NASM obtained by properly 2-coloring the vertices of the graph $\Gamma(A)$ associated to the array $A$ in the previous example.
\[M=
\begin{pmatrix}
0& 1&-1& 0\\
1&-1& 1&-1\\
0& 1& 0& 0
\end{pmatrix}
\]
\end{ex}

Given multisets of row and column weights, the problem of
determining whether there exists a convex binary array having these row and column weights is still open and belongs to the area of discrete tomography: it is indeed closely related to the reconstruction of ``$hv$-convex polyominoes'' from their horizontal and vertical projections (see, for instance, \cite{DLetalia, DuFro}). Note that without the convexity constraint, the existence problem is completely solved by the Gale--Ryser theorem.

\section{A composition construction}
We recall that $\Phi(m,n)$ and $\Psi(m,n)$ denote the sets of all $m\times n$ arrays with entries from $\{0,1\}$ and $\{0, \pm 1\}$, respectively.
Given a matrix $A\in \Psi(m,n)$, we denote by $f_{i}(A)$ ($i\in \Z_4$) the sequence associated to $A$
defined as follows:
\begin{enumerate}
  \item $f_{0}(A) = (a_1, \ldots, a_m)$, where each $a_i$ is the first nonzero entry in the $i$-th row of $A$ if there is any, otherwise, it is zero;
  \item $f_{2}(A) = (a_1, \ldots, a_m)$, where each $a_i$ is the last nonzero entry in the $i$-th row of $A$ if there is any, otherwise, it is zero;
  \item $f_{1}(A) = f_{0}(A^t)$ and $f_{3}(A) = f_{2}(A^t)$.
\end{enumerate}
The \emph{frame} of $A$ (a concept introduced in \cite{MT}) is the sequence 
\[\mathcal{F}(A)=(f_{0}(A), f_{1}(A), f_{2}(A), f_{3}(A)).
\] 
\begin{ex} Here we have a NASM $A$ and a representation of its frame,
\begin{center}
$ A = 
  \begin{pmatrix}
    + &  0 & - & +\\
    - &  + & 0 & -\\
    0 &  - & + & 0\\
  \end{pmatrix}$ \;\;\;\;\; 
\begin{tabular}{c|cccc|c}
\multicolumn{1}{c}{$f_{1}(A)$} & \multicolumn{1}{c}{$+$} &  $+$ & $-$ & $+$ & $f_{2}(A)$\\ \cline{1-5}
$+$ &  $+$ &    0       & $-$ & $+$ & $+$\\ 
$-$ &  $-$ &  $+$       &   0 & $-$ & $-$\\ 
$-$ &     0      &  $-$ & $+$ & 0   & $+$\\ \cline{2-6}
$f_{0}(A)$ &  $-$ &  $-$ & $+$ & \multicolumn{1}{c}{$-$} & \multicolumn{1}{c}{$f_{3}(A)$}\\ 
\end{tabular}
\end{center}
where
$f_{0}(A) = (+, -, -)$,
$f_{2}(A) = (+, -, +)$,
$f_{1}(A) = (+, +, -, +)$,
$f_{3}(A) = (-, -, +, -)$.
\end{ex}

We also denote by $\mathcal{F}_0(A)$ and $\mathcal{F}_1(A)$ the two arrays of size $m\times 2$ and $2\times n$, respectively, defined as follows:
\[
  \mathcal{F}_0(A)= (f_0(A)^t, f_2(A)^t) \;\;\;\text{and}\;\;\; \mathcal{F}_1(A)= (f_1(A), f_3(A)) 
\]

Given a sequence $\hh=(h_1, \ldots, h_n) \in \mathbb{N}_0^n$, set
\[
(-1)^\hh = 
((-1)^{h_1}, \ldots, (-1)^{h_n})]\;\;\;\text{and}\;\;\;
\hh+1 = (h_1+1, \ldots, h_n+1).
\]
The following lemma is straightforward.
\begin{lemma}\label{lem}
If $A\in NASM(m,n; \hh,\kk)$, where $\hh=(h_1, \ldots, h_m)$ and $\kk=(k_1, \ldots, k_n)$, then
\[
  f_0(A) = (-1)^{\hh+1} f_{2}(A),\;\;\; \text{and}\;\;\;
  f_1(A) = (-1)^{\kk+1} f_{3}(A).
\] 
\end{lemma}

Given two $m\times n$ arrays, say $M_1$ and $M_2$, we denote by  $M_1\circ M_2$ the 
Hadamard-Schur (entrywise) product of $M_1$ and 
$M_2$. Furthermore, letting $A=(a_{ij}) \in \Psi(m,n)$, we say that $A$ is even (resp. odd) if each of its rows and columns has even (resp. odd) weight. Also, we say that 
$a_{ij}$ precedes $a_{h,k}$ if one of the following conditions hold:
\begin{enumerate}
  \item $i=h$, $a_{ij}\neq 0\neq a_{ik}$ and $a_{i\ell}=0$ for every $\ell\in[j+1,k-1]$, or
  \item $j=k$, $a_{ij}\neq0\neq a_{hj}$ and $a_{\ell j}=0$ for every $\ell\in[i+1,h-1]$.
\end{enumerate}

\begin{theorem}\label{main2}
Let $\mathbb{A}=(A_{ij})$ be an $m \times n$ block matrix where each block  
$A_{ij}\in \Psi(m_i,n_j)$ and let 
$B=(b_{ij})\in \Psi(m,n)$. Then, $\mathbb{A}\circ B$ is a NASM if and only if the following conditions hold.
\begin{enumerate}
  \item each $A_{i,j}$ is a NASM,
  \item $b_{i,\alpha}f_2(A_{i,\alpha}) = -b_{i,\beta}f_0(A_{i,\beta})$, whenever 
  $b_{i,\alpha}$ precedes $b_{i,\beta}$,
  \item $b_{\alpha, j}f_3(A_{\alpha,j}) = -b_{\beta, j}f_1(A_{\beta, j})$, whenever 
  $b_{\alpha, j}$ precedes $b_{\beta, j}$.
\end{enumerate}
\end{theorem}
\begin{proof}
Since each block $A_{ij}$ is alternating, it remains only to verify that the compatibility conditions 2 and 3 across adjacent blocks ensure the alternation of signs in the rows and columns of $\mathbb{A}\circ B$. More precisely, condition~(2) guarantees that whenever two nonzero entries in the same row of $B$ occur consecutively, the last nonzero entries of the former block and the corresponding first nonzero entries of the latter block have opposite signs; hence the nonzero entries in each row of $\mathbb{A}\circ B$ alternate in sign. Similarly, condition~(3) ensures that the nonzero entries in each column of $\mathbb{A}\circ B$ alternate in sign. Therefore, $\mathbb{A}\circ B$ is a NASM.
The converse implication is immediate.
\end{proof}

It is not difficult to see that every NASM arises from the construction of Theorem~\ref{main2}. As an immediate consequence, we obtain the following corollary.

\begin{cor}\label{cor}
Let $\mathbb{A}=(A_{ij})$ be an $m \times n$ block matrix where each block  
$A_{ij}\in \NASM(m_i,n_j)$ and let 
$B=(b_{ij})\in \Psi(m,n)$. Also, assume that the $A_{ij}$s have the same frame. Then, $\mathbb{A} \circ B$ is an alternating sign matrix whenever one of the following conditions hold:
\begin{enumerate}
  \item  each $A_{ij}$ is even and $B\in \Phi(m,n)$;
  \item each $A_{ij}$ is odd and $B\in \NASM(m,n)$;
\end{enumerate}
\end{cor}
\begin{proof} By assumption, all $A_{ij}$s have the same frame. 
We first deal with the case where the $A_{ij}$s are all even and $B=(b_{ij})$ is binary. In this case, each $b_{i,j}=0,1$ and, by Lemma \ref{lem}, we have that 
\begin{enumerate}
\item $f_2(A_{i,\alpha}) = -f_0(A_{i,\alpha}) =  -f_0(A_{i,\beta})$ whenever 
$b_{i,\alpha}=1$ precedes $b_{i,\beta}=1$, 
\item 
$f_3(A_{\alpha,j}) = -f_1(A_{\alpha,j}) =-f_1(A_{\beta, j})$, whenever 
  $b_{\alpha, j}=1$ precedes $b_{\beta, j}=1$. 
\end{enumerate}  
  The assertion then follows from Theorem \ref{main2}.
  
  Now, assume that the $A_{ij}$s are all odd and $B=(b_{ij})$ is a NASM. By Lemma \ref{lem}, we have that 
\begin{enumerate}
\item $f_2(A_{i,\alpha}) = f_0(A_{i,\alpha}) =  f_0(A_{i,\beta})$ whenever 
$b_{i,\alpha}=\pm1$ precedes $b_{i,\beta}=\mp1$, 
\item 
$f_3(A_{\alpha,j}) = f_1(A_{\alpha,j}) =f_1(A_{\beta, j})$, whenever 
  $b_{\alpha, j}=\pm1$ precedes $b_{\beta, j}=\mp1$. 
\end{enumerate}  
  As before, the assertion follows from Theorem \ref{main2}.
\end{proof}

The following corollary is a specialization of the previous one.

\begin{cor}
Assume there is a $\NASM(m,n; \hh,\kk)$ and  let
$B\in \Psi(p,q; \uu, \vv)$. 
Then, there exists a $\NASM(mp, nq; (h_i u_\alpha), (k_jv_\beta))$ in each of the following cases:
\begin{enumerate}
  \item  $\hh$ and $\kk$ are even sequences and $B$ is binary, or
  \item $\hh$ and $\kk$ are odd sequences and $B$ is a $\NASM$
\end{enumerate}
(where $h_i, k_i, u_i, v_i$ are the $i$-th entries of 
$\hh,\kk,  \uu, \vv$, respectively).
\end{cor}
\begin{proof}
  It is enough to consider the $p\times q$ block-matrix $A$, whose blocks are copies of a $\NASM(m,n; \hh,\kk)$, and then apply Corollary \ref{cor}.
\end{proof}

We complete this section with an example that constructs NASMs based on the previous two corollaries. 

\begin{ex}
Let 
$A=
\begin{pmatrix}
 1&-1\\
-1& 1
\end{pmatrix}
$ and
$
B=
\begin{pmatrix}
1&1\\
0&1
\end{pmatrix}
$. Note that $A$ is a NASM with even row and column weights and $B$ is a binary array.
Replacing each $1$ of $B$ by $A$ and each $0$ by a zero block gives the following NASM:
\[
\begin{pmatrix}
 1&-1& 1&-1\\
-1& 1&-1& 1\\
 0& 0& 1&-1\\
 0& 0&-1& 1
\end{pmatrix}.
\]
Similarly, let
$A=
\begin{pmatrix}
 1&-1&1\\
-1& 1&-1\\
 1&-1&1\\
\end{pmatrix}
$ and
$
B=
\begin{pmatrix}
1&-1\\
0&1
\end{pmatrix}
$. Note that $A$ is a NASM with odd row and column weights and $B$ is a NASM.
This time, replacing each $1$ of $B$ by $A$, each $-1$ by $-A$ and each $0$ by a zero block gives the following NASM:
\[
\begin{pmatrix}
 1&-1& 1&-1& 1&-1\\
-1& 1&-1& 1&-1& 1\\
 1&-1& 1&-1& 1&-1\\
 0& 0& 0& 1&-1& 1\\
 0& 0& 0&-1& 1&-1\\
 0& 0& 0& 1&-1& 1  
\end{pmatrix}.
\]
\end{ex}

\section{Concluding remarks and open problems}
In this paper, we investigated the existence problem for near alternating sign matrices with prescribed row and column weight sequences.

Our first main contribution is the graph-theoretic characterization given in Theorem~\ref{thm:bipartite}, which shows that the existence of a $\mathrm{NASM}(m,n;\h,\k)$ is equivalent to the existence of a binary array in $\Phi(m,n;\h,\k)$ whose associated graph is bipartite. This characterization immediately implies that the necessary conditions of Lemma~\ref{nec} are not sufficient in general when the row and column weight sequences are prescribed in a fixed order.

On the other hand, the example following Theorem~\ref{thm:bipartite} suggests that the situation may change substantially when only the multisets of row and column weights are prescribed. This led us to formulate Conjecture~\ref{conj:main}, which asks whether the conditions of Lemma~\ref{nec} are sufficient for the existence of a $\mathrm{NASM}(m,n;[\h],[\k])$.

Our second contribution is the composition construction developed in Section~3. Using the notion of frame, we obtained a general composition result, Theorem~\ref{main2}, that allows one to construct larger NASMs from smaller ones. As applications, we derived the existence of infinite families of NASMs.

Several interesting questions remain open.

\begin{enumerate}
\item Prove or disprove Conjecture~\ref{conj:main}.

\item Characterize the pairs $(\h,\k)$ for which there exists a
$\mathrm{NASM}(m,n;\h,\k)$ with the prescribed ordering.
In other words, this means to determine whether there exists an analogue of the Gale--Ryser theorem for NASMs. This is also equivalent to characterizing projection vectors (the row and column weight sequences) admitting a binary realization $A$ whose graph $\Gamma(A)$ is bipartite.

\item A stronger version of the previous problem is to characterize the quadruples $(L,T,\h,\k)$ for which there exists a
$\mathrm{NASM}(m,n;\h,\k)$, say $A$, whose left and top frame components satisfy
$f_0(A)=L$ and $f_1(A)=T$.
Notice that, by Lemma~\ref{lem}, the remaining frame components $f_2(A)$ and $f_3(A)$ are completely determined by $f_0(A)$, $f_1(A)$, $\h$, and $\k$.

\item Investigate the existence of convex arrays with prescribed multisets of row and column weights. By Theorem~\ref{thm:convex}, any positive result in this direction immediately yields new existence results for NASMs.

\item Develop further recursive and composition constructions based on frames, and determine which classes of NASMs can be generated from a finite collection of basic building blocks.
\end{enumerate}

\end{document}